\documentclass[a4paper, 10pt, notitlepage]{article}

\usepackage{amsthm, amsmath, amssymb, latexsym}
\usepackage{mathrsfs,cite}
\usepackage[multiple]{footmisc}
\usepackage{graphicx}

\theoremstyle{plain}
\newtheorem{theorem}{Theorem}[section]
\newtheorem{lemma}[theorem]{Lemma}

\newtheorem{proposition}{Proposition}[section]

\theoremstyle{definition}

\theoremstyle{remark}

\DeclareMathOperator{\co}{co}
\DeclareMathOperator{\cl}{cl}

\DeclareMathOperator{\sign}{sign}

\author{M.V. Dolgopolik and A.L. Fradkov}
\title{Speed-Gradient Control of the Brockett Integrator}

\begin{document}

\maketitle

\begin{abstract}
A nonsmooth extension of the Speed-Gradient (SG) algorithms in finite form is proposed. The conditions ensuring control
goal (convergence of the goal function to zero) are established. A new algorithm is applied to almost global
stabilization of the Brockett integrator that has become a popular benchmark for nonsmooth and discontinuous
algorithms. It is proved that the designed control law stabilizes the Brockett integrator for any initial point that
does not lie on the x3-axis. Besides, it is shown that Speed-Gradient algorithm ensures stabilization with arbitrarily
small control level. An important feature of the proposed control is the fact that it is continuous along trajectories
of the closed-loop system.
\end{abstract}

\section{Introduction}

The Brockett integrator or nonholonomic integrator 
\begin{equation}
\label{BI}
\dot x_1=u_1, \quad \dot x_2=u_2, \quad \dot x_3=x_1u_2-x_2u_1
\end{equation}
was introduced in the seminal paper \cite{Brockett}.  Since it was proved by Brockett that 
there exist no smooth time-invariant state feedback stabilizing (\ref{BI}) in the origin, it
has become a paradigmatic example of systems where smooth feedback fails. This results was further extended to a large
class of discontinuous state feedbacks by Ryan \cite{Ryan}. It also was, in a sense, the starting point of application
of differential geometry, Lie groups and Lie algebra methods to nonlinear control that lead
to the creation of what is now known under the name of geometric control theory. The development of the geometric
control theory started in the 1970s \cite{Brockett_GeomCont}. By now, it has grown into a mature area with strong
machinery \cite{Bonnard,Agrachev,Bullo,Stefani}. Since Brockett integrator is beautiful and seemingly simple system, it
has become a benchmark example for nonlinear control methods.  

During several decades many authors have been making efforts to apply their approaches for control of the Brockett
integrator. New algorithms for (\ref{BI}) were designed via the invariant manifold technique in \cite{Khennouf} and via 
discontinuous transformations in \cite{Astolfi}. The sliding mode control was applied in \cite{Bloch}.  
A family of discontinuous control laws was derived in \cite{Astolfi2}.
A ``sample-and-hold'' approach based on nonsmooth control Lyapunov functions was proposed in \cite{Clarke}.
Logic based switching was applied in \cite{Morse}. Methods of optimal control theory was utilised in \cite{Vdovin}. 
A hybrid control law was designed in \cite{Prieur, Prieur2}. An impulsive control was studied in \cite{Kovalev}. An
interesting general approach, using isospectral flows, to the stabilization of a class of nonholonomic systems that
includes the Brockett integrator was developed in \cite{BlochDrakunovKinyon}

In this paper we make an attempt to control the Brockett integrator by means of the Speed-Gradient method.
The Speed-Gradient (SG) method was proposed in the end of the 1970s as a general framework for design of control,
adaptation, identification algorithms for nonlinear systems \cite{F79}. Since then it was extended in different
directions \cite{F86,FMN99} and applied to a variety of problems in physics and mechanics
\cite{FS92,FP96,Selivanov12}. An intimate relation between applicability of SG-method and passivity of
controlled system was established \cite{FGHP95}. In the special case of affine controlled system the SG-algorithms
encompass Jurdjevic-Quinn (LgV) algorithms \cite{JQ78}.

Standard procedure of SG-algorithms derivation requires differentiation of the goal function along trajectories of the
controlled system. However, in many cases the right hand sides of the system model are nonsmooth. Sometimes it may be
profitable to introduce nonsmooth and even discontinuous terms into control algorithms in order to provide the desired
system dynamics, e.g. finite time convergence. Therefore there is a need for a more general framework for design and
analysis of SG-like algorithms in a general nonsmooth case. A first extension of SG-methods to nonsmooth case has been
made in \cite{DF}. It should be noted that the assumptions on the controlled system and the goal function that we
use in this article (see~Theorem~\ref{Thm_LinCase} below) are different from the ones in \cite{DF}. Furthermore, in
\cite{DF} the nonsmooth Speed-Gradient algorithm was applied only to a linear controlled system, while the main goal of
the present article is to apply this algorithm to the stabilization problem for the Brockett integrator.

In Section~\ref{NonSm_SGA} of this paper we introduce a further result on stabilization ability of nonsmooth
pseudogradient methods. In Section~\ref{Section_BI} a nonsmooth SG-algorithm for stabilization of (\ref{BI}) is derived,
and stability conditions are established. As a further application of nonsmooth Speed-Gradient methods, we consider
the energy control problem for a vibrating string in Section~\ref{Section_VibrString}. Necessary preliminary material is
presented in Section~\ref{Prelim}.

\section{Preliminaries}
\label{Prelim}

In this section, we recall some notions from nonsmooth analysis \cite{Demyanov} that are used in the sequel. Denote by
$|\cdot|$ the Euclidean norm in $\mathbb{R}^n$, and denote $\mathbb{R}_+ = [0, + \infty)$.

Let a real-valued function $f$ be defined in a neighbourhood of a point $x \in \mathbb{R}^n$. The function $f$ is
called \textit{Hadamard directionally differentiable} at the point $x$ if for any $h \in \mathbb{R}^n$ there exists 
the finite limit
$$
  f'(x; h) = \lim_{[\alpha, h'] \to [+0, h]} \frac{f(x + \alpha h') - f(x)}{\alpha}.
$$
The function $f'(x; \cdot)$ is called the \textit{Hadamard directional derivative} of $f$ at $x$. Note that
there exists an elaborate calculus of Hadamard directional derivatives \cite{Demyanov}. Observe also that if $n = 1$,
then the quantity $f'(x; 1)$ coincides with the right-hand side derivative of $f$ at $x$ that is denoted by
$D_+ f(x)$. 

Recall that the function $f$ is referred to as \textit{Hadamard superdifferentiable} at the point $x$, if $f$ is
Hadamard directionally differentiable at this point, and there exists a convex compact set 
$\overline{\partial} f(x) \subset \mathbb{R}^n$ such that
$$
  f'(x; h) = \min_{v \in \overline{\partial}f(x)} v^T h \quad \forall h \in \mathbb{R}^n.
$$
The set $\overline{\partial} f(x)$ is called the (Hadamard) \textit{superdifferential} of $f$ at $x$. The most common
example of a Hadamard superdifferentiable function is the composition of a concave function and a differentiable vector
function.

It is easy to see that if functions $f_i$, $i \in I = \{ 1, \ldots, k \}$ are Hadamard superdifferentiable at a point 
$x \in \mathbb{R}^n$, then for any $\alpha_i \ge 0$, $i \in I$ the functions $\sum_{i \in I} \alpha_i f_i$ and 
$\min_{i \in I} f_i$ are Hadamard superdifferentiable at $x$ as well. Note that any semiconcave function \cite{Clarke}
is Hadamard superdifferentiable.

\section{Nonsmooth Speed-Gradient Algorithm}
\label{NonSm_SGA}

Consider the controlled system
\begin{equation} \label{ContrSys}
  \dot{x} = F(x, u, t), \quad t \ge 0, 
\end{equation}
where $x \in \mathbb{R}^n$ is the vector of the system state, and $u \in \mathbb{R}^m$ is the control. We assume that
the function $F \colon \mathbb{R}^n \times \mathbb{R}^m \times \mathbb{R}_+ \to \mathbb{R}^n$ is continuous. A solution
of (\ref{ContrSys}), even in the case of a discontinuous control law, is understood to be an absolutely continuous
function satisfying (\ref{ContrSys}) for almost all $t$ in its domain.

We pose the general control problem as finding the control law 
$$
  u(t) = U\{ x(s), u(s) \colon 0 \le s \le t \}
$$
which ensures the control objective 
$$
  Q(x(t), t) \le \Delta \text{ when } t \ge t_*,
$$
where $Q(x, t)$ is a nonnegative \textit{goal function} defined on $\mathbb{R}^n \times \mathbb{R}_+$, $\Delta \ge 0$ is
some pre-specified threshold, and $t^*$ is the time instant at which the control objective is achieved. The objective
can be formulated also as 
$$
  \limsup_{t \to \infty} Q(x(t), t) \le \Delta,
$$
which does not specify the value of $t^*$. In the special case $\Delta = 0$ the control objective takes the form
\begin{equation} \label{ContrGoal}
  \lim_{t \to \infty} Q(x(t), t)) = 0,
\end{equation}
i.e. the objective is to stabilize the system (\ref{ContrSys}) with respect to the goal function $Q$.

The formulation of the control problem that we use encompasses various control problems, such as partial stabilization,
control of system energy, identification and adaptive control (see the discussion, as well as various examples and
applications, in \cite{FMN99}). In particular, if one takes a control Lyapunov function $V(x)$ of the system
(\ref{ContrSys}) as the goal function $Q(x, t)$, then the goal (\ref{ContrGoal}) is closely related to asymptotic
stability of (\ref{ContrSys}). However, we underline that possible goal functions $Q(x, t)$ are neither
exhausted by nor reduced to control Lyapunov functions (see~Section~\ref{Section_VibrString} below and examples in
\cite{FMN99}).

In order to design a control algorithm suppose that the function $Q$ is locally Lipschitz continuous and Hadamard
directionally differentiable. Choose a function $\omega(x, u, t)$ of the form 
$$
  \omega(x, u, t) = g(x, t)^T u,
$$
where $g(x, t) \colon \mathbb{R}^n \times \mathbb{R}^+ \to \mathbb{R}^m$ is a given function such that
\begin{equation} \label{InequalForOmega}
  Q'(x, t; F(x, u, t), 1 ) \le \omega(x, u, t) \quad 
  \forall x \in \mathbb{R}^n, u \in \mathbb{R}^m, t \in \mathbb{R}_+.
\end{equation}
Note that $g(x, t) = \nabla_u \omega(x, u, t)$, where $\nabla_u \omega(x, u, t)$ is the gradient of the function
$u \to \omega(x, u, t)$.

\noindent{\em Remark~1}.~The existence of a function $\omega(x, u, t)$ of the form $\omega(x, u, t) = g(x, t)^T u$ that
satisfies (\ref{InequalForOmega}) is the basic assumption on the system (\ref{ContrSys}) and the goal function $Q(x, t)$
that we implicitly make throughout this article. This assumption is valid, in particular, in the case when the function
$F$ is linear in $u$ (i.e. when it has the form $F(x, u, t) = f(x, t) u$), and the function $Q$ is Hadamard
superdifferentiable and does not depend on $t$. Indeed, in this case one can define $\omega(x, u, t) = v(x)^T f(x, t) u$
for any function $v(x)$ such that $v(x) \in \overline{\partial} Q(x)$ for all $x \in \mathbb{R}^n$.

Take the control algorithm in the form 
$$
  u = - \Gamma g(x, t),
$$
where $\Gamma$ is a positive definite gain matrix. We will also consider a more general control algorithm
\begin{equation} \label{PseudoSG}
  u = \gamma \psi(x, u, t) 
\end{equation}
where $\gamma > 0$ is a scalar gain, and the vector function $\psi$ satisfies the ``acute angle'' condition:
$g(x, t)^T \psi(x, u, t) \le 0$ for any $x \in \mathbb{R}^n$, $u \in \mathbb{R}^m$ and $t \in \mathbb{R}_+$.
The algorithm of the form (\ref{PseudoSG}) is a generalization of the so-called Speed-Pseudo\-gradient algorithms (see
\cite{FMN99}). Furthermore, if one takes a control Lyapunov function $V(x)$ of the system (\ref{ContrSys}) as the goal
function $Q(x, t)$, then the control law (\ref{PseudoSG}) is a feedback of steepest descent type for $V$
(see~\cite{Clarke}). However, since, as it was mentioned above, the function $Q(x, t)$ need not be a control Lyapunov
function, the control algorithm (\ref{PseudoSG}) is not reduced to a feedback of steepest descent type in the
general case.

It should be noted that (\ref{PseudoSG}) is an \textit{equation} with respect to the control variable
$u$; in other words, the equality (\ref{PseudoSG}) defines the control law $u = u(x, t, \gamma)$ \textit{implicitly}.
Therefore, in order to implement the algorithm of the form (\ref{PseudoSG}) one should be able to efficiently solve
this equation, i.e. one should be able either to obtain an explicit expression for $u(x, t, \gamma)$ or to efficiently
solve this equation numerically. However, it should be noted that in many applications either the function $\psi$ does
not depend on $u$ (see examples below) or a solution of (\ref{PseudoSG}) can be found analytically.

Let us discuss the performance of the control systems with the proposed control algorithm (\ref{PseudoSG}). 
The following theorem holds true.

\begin{theorem} \label{Thm_LinCase}
Let $C \subset \mathbb{R}^n$ be a given set, and the following assumptions be valid:
\begin{enumerate}
\item{for any $\gamma > 0$, $x \in \mathbb{R}^n$ and $t \ge 0$ there exists a solution $u = \kappa(x, t, \gamma)$
of equation (\ref{PseudoSG}), and the function $\kappa$ is locally bounded in $x$ uniformly in $t$;}

\item{an absolutely continuous solution of the system (\ref{ContrSys}), (\ref{PseudoSG}) exists for all $t \ge 0$ and 
$x(0) \in \mathbb{R}^n \setminus C$, and $x(t) \notin C$ for any $t \in \mathbb{R}_+$;}

\item{the function $Q(x, t)$ is radially unbounded, i.e.
$$
  \inf_{t \ge 0} Q(x, t) \to \infty \quad \text{as} \quad |x| \to \infty,
$$
and nonnegative;
}
\item{for any $\Delta > 0$ and $r > 0$ there exists $a > 0$ such that $|g(x, t)| \ge a$ for all 
$x \in \mathbb{R}^n \setminus C$ and $t \in \mathbb{R}_+$ such that $Q(x, t) \ge \Delta$ and $|x| \le r$;
}
\item{there exists a continuous function $\rho \colon \mathbb{R}_+ \to \mathbb{R}_+$ such that $\rho(s) = 0$ if and
only if $s = 0$, and $g(x, t)^T \psi(x, u, t) \le - \rho( | g(x, t) | )$ for all $x \in \mathbb{R}^n$, 
$u \in \mathbb{R}^m$ and $t \in \mathbb{R}_+$.
}
\end{enumerate}

Then for any $x(0) \in \mathbb{R}^n \setminus C$ and $\gamma > 0$ a solution of (\ref{ContrSys}), (\ref{PseudoSG}) is
bounded on $\mathbb{R}_+$ and the control goal (\ref{ContrGoal}) is achieved, i.e.
$$
  \lim_{t \to \infty} Q(x(t), t)) = 0.
$$
\end{theorem}

\begin{proof}
Let $x(t)$ be a solution of the system (\ref{ContrSys}), (\ref{PseudoSG}). Introduce the Lyapunov
function $V(x, t) = Q(x, t)$, and define $V_0(t) = V(x(t), t)$. The function $Q$ is Hadamard directionally
differentiable, and the function $x(t)$ is a.e. differentiable, as a solution of a differential equation. Therefore by
the chain rule for directional derivatives (Theorem~I.3.3, \cite{Demyanov}) for a.e. $t \in \mathbb{R}_+$ there exists
the right-hand side derivative $D_+ V_0(t)$ of the function $V_0$ that has the form
$$
  D_+ V_0(t) = Q'(x(t), t; F(x(t), u, t), 1) \le \omega( x(t), u, t).
$$
Applying assumption 5 and the fact that $\omega(x, u, t) = g(x, t)^T u$ one obtains that
\begin{equation} \label{LyapFunc_LinCase}
  D_+ V_0(t) \le \gamma g(x, t)^T \psi(x, u, t) \le - \gamma \rho( |g(x, t)| ) \le 0.
\end{equation}
Note that the function $V_0$ is absolutely continuous as the composition of the locally Lipschitz continuous function
$Q(x, t)$ and the absolutely continuous function $x(t)$. Hence taking into account (\ref{LyapFunc_LinCase}) one gets
that the function $V_0(t)$ is nonincreasing, which implies the boundedness of $x(t)$ due to the radial
unboundedness of $Q$, and the boundedness of the control $u$ due to the local boundedness in $x$ uniformly in $t$ of
the function $u = \kappa(x, t, \gamma)$.

Choose an arbitrary $\Delta > 0$, and denote $T_{\Delta} = \{ t \ge 0 \colon Q(x(t), t) \ge \Delta \}$. Observe that
the set $T_{\Delta}$ is connected due to the fact that the function $V_0(t) = Q(x(t), t)$ is nonincreasing. 
From assumption 4 it follows that there exists $a > 0$ such that $| g(x, t) | > a$ for all $t \in T_{\Delta}$.
Hence with the use of (\ref{LyapFunc_LinCase}) one gets that $D_+ V_0(t) \le - \gamma \rho(a) < 0$ for any 
$t \in T_{\Delta}$. Consequently, $\sup T_{\Delta} \le V_0(0) / \gamma \rho(a)$, and
$$
  V_0(t) = Q(x(t), t) < \Delta \quad \forall t > \sup T_{\Delta}.
$$
Hence and from the fact that $\Delta > 0$ is arbitrary it follows that (\ref{ContrGoal}) holds true.
\end{proof}

\noindent{\em Remark~2}. Let all assumptions of the theorem above be valid, and suppose that the function 
$\psi(x, u, t)$ is bounded for any $x \in \mathbb{R}^n$, $u \in \mathbb{R}^m$ and $t \in \mathbb{R}_+$. Then the control
goal (\ref{ContrGoal}) can be achieved for an arbitrarily small control input. Indeed, by choosing sufficiently small 
$\gamma > 0$, one can obtain that the inequality $|u| = \gamma |\psi(x, u, t)| < \varepsilon$ holds true for an
arbitrarily small prespecified $\varepsilon > 0$.

The following lemma allows one to slightly improve Theorem~\ref{Thm_LinCase}.

\begin{lemma}
Let assumptions 1--3 of Theorem~\ref{Thm_LinCase} hold true, and let 
$$
  g(x, t)^T \psi(x, u, t) \le 0 \quad \forall u \in \mathbb{R}^m
$$
for any $x \in \mathbb{R}^n$ and $t \in \mathbb{R}_+$ such that $Q(x, t) > 0$. Suppose that $x(t)$ is a solution of the
system (\ref{ContrSys}), (\ref{PseudoSG}) with $x(0) \in \mathbb{R}^n \setminus C$ such that $Q(x(T), T) = 0$ for some 
$T \ge 0$. Then $Q(x(t), t) = 0$ for all $t \ge T$.
\end{lemma}

\begin{proof}
Arguing by reductio ad absurdum, suppose that there exists $t_0 > T$ such that $Q(x(t_0), t_0) > 0$.
Denote 
$$
  \tau = \sup\big\{ t \in [T, t_0] \colon Q(x(t), t) = 0 \big\}.
$$
Then $T \le \tau < t_0$, $Q(x(\tau), \tau) = 0$ and for any $t \in (\tau, t_0]$ one has $Q(x(t), t) > 0$. Hence for
a.e. $t \in (\tau, t_0]$ one has
$$
  D_+ V_0(t) = Q'(x(t), t; F(x(t), u, t), 1) \le \omega(x(t), u, t) := \gamma g(x, t)^T \psi(x, u, t) \le 0,
$$
where $V_0(t) = Q(x(t), t)$. Consequently, the function $V_0$ is nonincreasing on $[\tau, t_0]$. Therefore
$Q(x(t_0), t_0) \le Q(x(\tau), \tau) = 0$, which contradicts the definition of $t_0$.
\end{proof}

\noindent{\em Remark~3}.~From the lemma above it follows that Theorem~\ref{Thm_LinCase} holds true in the case when 
the inequality
$$
  Q'(x, t; F(x, u, t), 1) \le \omega(x, u, t) \quad \forall u \in \mathbb{R}^m
$$
(see~(\ref{InequalForOmega})) is satisfied only for all $x \in \mathbb{R}^n$ and $t \in \mathbb{R}_+$ such that 
$Q(x, t) > 0$. Indeed, if $Q(x(t), t) > 0$ for all $t \in \mathbb{R}_+$, then arguing in the same way as in the
proof of Theorem~\ref{Thm_LinCase} one obtains that $Q(x(t), t) \to 0$ as $t \to \infty$. On the other hand, if
$Q(x(T), T) = 0$ for some $T \ge 0$, then arguing in the same way as in the proof of the previous lemma one
gets that $Q(x(t), t) = 0$ for all $t \ge T$, which implies the desired result.

\section{Stabilization of the Brockett Integrator}
\label{Section_BI}

\subsection{Problem Formulation}

Let us apply the theory discussed above to the construction of an arbitrarily small stabilizing feedback control for
the Brockett integrator (\ref{BI}). Since there is no continuous feedback control that stabilizes this system
\cite{Brockett}, the standard Speed-Gradient algorithms cannot be applied in this case. That is why we utilise the
nonsmooth version of SG-algorithm developed in this paper.

As it was mentioned above, the Brockett integrator is expressed as (\ref{BI}). We impose the additional 
constraint on control $u_1^2 + u_2^2 \le \varepsilon$, where $\varepsilon > 0$ is arbitrary. Being inspired with 
the ideas of Clarke \cite{Clarke}, introduce the goal function $Q(x)$ as follows
$$
  Q(x) = \left(\sqrt{x_1^2 + x_2^2} - |x_3| \right)^2 + x_3^2 = 
  x_1^2 + x_2^2 + 2 x_3^2 - 2 |x_3| \sqrt{x_1^2 + x_2^2}.
$$
Note that the function $Q$ is radially unbounded, and $Q(x) = 0$ iff $x = 0$. It was shown in \cite{Clarke} that $Q$ is
a control Lyapunov function for the Brockett integrator.

\subsection{Feedback Construction}

Let us apply the algorithm (\ref{PseudoSG}) to the construction of a control law. For the sake of convenience,
denote $\sigma(x) = \sqrt{x_1^2 + x_2^2}$. The function $Q$ is locally Lipschitz continuous and Hadamard directionally
differentiable. Its directional derivative has the form
$$
  Q'(x; h) = 2 x_1 h_1 + 2 x_2 h_2 + 4 x_3 h_3
  - 2 |x_3| \frac{( x_1 h_1 + x_2 h_2 )}{\sigma(x)} - 2 \sign( x_3 ) \sigma(x) h_3	
$$
in the case $x_3 \ne 0$ and $\sigma(x) \ne 0$, and
$$
  Q'(x, h) = \begin{cases}
    2 x_2 h_1 + 2 x_2 h_2 - 2 |h_3| \sigma(x), \text{ if } x_3 = 0, \\
    4 x_3 h_3 - 2 |x_3| \sqrt{h_1^2 + h_2^2}, \text{ if } \sigma(x) = 0.
  \end{cases}
$$

Let $x_3 \ne 0$ and $\sigma(x) \ne 0$. Then the function $Q'(x; \cdot)$ is linear, i.e. $Q$ is differentiable.
Therefore in this case define
\begin{multline*}
  \omega(x, u) = Q'(x)^T F(x, u) = 2 x_1 u_1 + 2 x_2 u_2 + 4 x_3 (x_1 u_2 - x_2 u_1) \\ 
  - 2 |x_3| (x_1 u_1 + x_2 u_2) / \sigma(x) - 2 \sign(x_3) \sigma(x) (x_1 u_2 - x_2 u_1),
\end{multline*}
where $F(x, u)$ is the right-hand side of (\ref{BI}).

Let now $x_3 = 0$. Then $Q$ is Hadamard superdifferentiable, and its superdifferential has the form
$$
  \overline{\partial} Q(x) = \co\{ (2 x_1, 2 x_2, 2 \sigma(x))^T, (2 x_1, 2 x_2, -2 \sigma(x))^T \},
$$
where ``$\co$'' stands for the convex hull. Note that $(2 x_1, 2 x_2, 0) \in \overline{\partial} Q(x)$, and define
\begin{equation} \label{BI_omega2}
  \omega(x, u) = (2 x_1, 2 x_2, 0)^T F(x, u) = 2 x_1 u_1 + 2 x_2 u_2 \ge Q'(x; F(x, u)).
\end{equation}
Finally, let $\sigma(x) = 0$. Then $Q$ is also Hadamard superdifferentiable and
$$
  \overline{\partial} Q(x) = \big\{ ( - 2 |x_3| v_1, - 2 |x_3| v_2, 4 x_3 )^T \colon |(v_1, v_2)| \le 1 \big\}.
$$
Therefore, choose $v = (v_1, v_2) \in \mathbb{R}^2$ such that $|v| = 1$, and define
\begin{multline} \label{BI_omega3}
  \omega(x, u) = ( - 2 |x_3| v_1, - 2 |x_3| v_2, 4 x_3 )^T F(x, u) \\
  = - 2 |x_3| (v_1 u_1 + v_2 u_2) \ge Q'(x; F(x, u))
\end{multline}
(recall that $\sigma(x) = 0$, i.e. $x_1 = x_2 = 0$). Note that the choice of $v$ can depend on $x_3$, i.e. one can
choose $v = v(x_3)$.

Define the control law as follows
\begin{equation} \label{GenContrL}
  u(x) = \gamma \psi(x), \quad \psi(x) = - |\nabla_u \omega(x, u)|^{-1} \nabla_u \omega(x, u).
\end{equation}
Thus, the control has the form 
\begin{equation} \label{ControlLaw}
  u(x) = \begin{cases}
    0, \text{ if } x = 0, \\
    - \gamma \sigma(x)^{-1} (x_1, x_2)^T, \text{ if } x_3 = 0, \sigma(x) \ne 0 \\
    \gamma v(x_3), \text{ if } \sigma(x) = 0, x_3 \ne 0. \\
    - \gamma |\nabla_u \omega(x, u)|^{-1} \nabla_u \omega(x, u), \text{ if } \sigma(x) \ne 0, x_3 \ne 0, \\
  \end{cases}
\end{equation}
where $\nabla_u \omega(x, u) = (\partial \omega / \partial u_1, \partial \omega / \partial u_2)$ and
\begin{gather} \label{PartDer1}
  \frac{\partial \omega}{\partial u_1} (x, u) = 2 x_1 - 4 x_2 x_3 - 
  2 \frac{|x_3| x_1}{\sigma(x)} + 2\sign(x_3) x_2 \sigma(x) \\
  \frac{\partial \omega}{\partial u_2} (x, u) = 2 x_2 + 4 x_1 x_3 - 
  2 \frac{|x_3| x_2}{\sigma(x)} - 2 \sign(x_3) x_1 \sigma(x). \label{PartDer2}
\end{gather}
in the case $\sigma(x) \ne 0$ and $x_3 \ne 0$.

Let us show that $|\nabla_u \omega(x, u)| \ne 0$ for any $x$ such that $\sigma(x) \ne 0$ and $x_3 \ne 0$. Indeed,
multiplying by $\sigma(x)$ in (\ref{PartDer1}) one gets that for any $x$ such that $x_1 \ne 0$ and $x_3 \ne 0$ 
the following holds true
$$
  \sigma(x) \frac{\partial \omega}{\partial u_1} (x, u) = 2 \sign(x_3) x_2 \sigma^2(x)
  + (2 x_1 - 4 x_2 x_3) \sigma(x) - 2 |x_3| x_1 \ne 0,
$$
since the discriminant of the quadratic equation
$$
  l(s) = 2 \sign(x_3) x_2 s^2 + (2 x_1 - 4 x_2 x_3) s - 2 |x_3| x_1 = 0
$$
has the form $D = 4 x_1^2 + 16 x_2^2 x_3^2 > 0$. Analogously, $\partial \omega / \partial u_2 \ne 0$ for all $x$ such
that $x_2 \ne 0$ and $x_3 \ne 0$. Thus, $|\nabla_u \omega(x, u)| \ne 0$ for any $x \in \mathbb{R}^3$ such that 
$x_3 \ne 0$ and $\sigma(x) \ne 0$. Hence the control law (\ref{ControlLaw}) is corretly defined.

\noindent{\em Remark~4}.~{(i) Observe that in the case $\sigma(x) \ne 0$ and $x_3 \ne 0$, the set of limit points of
$u(x)$ (see (\ref{ControlLaw})--(\ref{PartDer2})) as $\sigma(x) \to 0$ is the circle of radius $\gamma$ centred at
the origin. Interestingly, according to the algorithm one defines $u(x)$, when $\sigma(x) = 0$, as an arbitrary element
of this circle.
}

\noindent{(ii) From the definition it follows that the control law (\ref{ControlLaw}) is a feedback of steepest descent
type for the control Lyapunov function $Q(x)$. However, we want to point out that we understand a solution of
a differential equation in the classical sense as opposed to the sample-and-hold sense in \cite{Clarke}, where a
different discontinuous stabilizing feedback of steepest descent type for the Brockett integrator was constructed.
Moreover, we will show that unlike the feedback controller proposed in \cite{Clarke}, the control law (\ref{ControlLaw})
is continuous along solutions of the closed-loop system for almost all initial points. 
}

\noindent{(iii) Note that the state feedback (\ref{ControlLaw}) is not upper semicontinuous as a set-valued mapping.
Hence it does not fall into the class of discontinuous state feedbacks \cite{Ryan} that do not stabilize Brockett
integrator.
}

\subsection{Properties of the Designed Control Law}

Let us verify that all assumption of Theorem~\ref{Thm_LinCase} hold true with 
$$
  C = \{ x \in \mathbb{R}^3 \colon \sigma(x) = 0, x_3 \ne 0 \}.
$$
Then we can conclude that the control law (\ref{ControlLaw}) stabilizes the Brockett integrator for any $\gamma > 0$ and
any initial point $x(0)$ that does not belong to the $x_3$-axis. Moreover, by choosing $\gamma = \sqrt{\varepsilon}$ and
taking into account the fact that $|u(x)| = \gamma$ (see~(\ref{GenContrL})) one obtains that the proposed control
satisfies the constraint $u_1(x)^2 + u_2(x)^2 \le \varepsilon$, i.e. it can be made arbitrarily small. 

Clearly, assumptions 1 and 3 of Theorem~\ref{Thm_LinCase} are satisfied. Furthermore, from the fact that 
$$
  \psi(x) = - |\nabla_u \omega(x, u)|^{-1} \nabla_u \omega(x, u),
$$
it follows that assumption 5 is satisfied with $\rho(s) \equiv s$. 

\begin{proposition}
Assumption 4 of Theorem~\ref{Thm_LinCase} is satisfied in the example under consideration.
\end{proposition}

\begin{proof}
Introduce a set-valued mapping $G \colon \mathbb{R}^3 \rightrightarrows \mathbb{R}^2$. Define 
$G(x) = \nabla_u \omega(x, u)$ if $\sigma(x) \ne 0$ and $x_3 \ne 0$, 
\begin{multline*}
  G(x) = \big\{ (2 x_1, 2 x_2)^T, (2 x_1 + 2 x_2 \sigma(x), 2 x_2 - 2 x_1 \sigma(x))^T, \\
  (2 x_1 - 2 x_2 \sigma(x), 2 x_2 + 2 x_1 \sigma(x))^T \big\}
\end{multline*}
if $x_3 = 0$, and $G(x) = \{ - 2|x_3| w \colon |w| = 1 \}$, if $\sigma(x) = 0$. By definition (see (\ref{BI_omega2})
and (\ref{BI_omega3})), one has $\nabla_u \omega(x, u) \in G(x)$ for any $x$. Furthermore, it is easy check that 
$0 \in G(x)$ if and only if $x = 0$.

As it was mentioned above, in the case $\sigma(x) \ne 0$ and $x_3 \ne 0$, the set of limit points of 
$\nabla_u \omega(x, u)$ (see (\ref{PartDer1}) and (\ref{PartDer2})) as $\sigma(x) \to 0$ is the circle of radius 
$2 |x_3|$ centred at the origin. Note also that in the same case the set of limit points of $\nabla_u \omega(x, u)$ as
$x_3 \to 0$ consists of two points: 
$$
  (2 x_1 + 2 x_2 \sigma(x), 2 x_2 - 2 x_1 \sigma(x))^T, \quad
  (2 x_1 - 2 x_2 \sigma(x), 2 x_2 + 2 x_1 \sigma(x))^T.
$$
Thus, in the case when $\sigma(x) = 0$ or $x_3 = 0$ the set $G(x)$ consists of $\nabla_u \omega(x, u)$ and all limit
points of $\nabla_u \omega(y, u)$ as $y \to x$. Therefore it is easy to verify that the set-valued mapping $G$ is upper
semicontinuous, i.e. for any $x$ and any open set $V$ such that $G(x) \subset V$ there exists $\delta > 0$ such that for
any $y$ with $|y - x| < \delta$ one has $G(y) \subset V$.

Arguing by reductio ad absurdum, suppose that assumption 4 does not hold true. Then there exists $\Delta > 0$, $r > 0$
and a sequence $\{ x^{(n)} \}$ such that
$$
  Q(x^{(n)}) \ge \Delta, \quad |x^{(n)}| \le r, \quad |\nabla_u \omega(x^{(n)}, u)| \le \frac{1}{n}.
$$
Consequently, there exists a subsequence, which we denote again by $\{ x^{(n)} \}$, converging to some 
$x^*$. Observe that $Q(x^*) \ge \Delta$ and $x^* \ne 0$ due to the facts that the function $Q$ is continuous and 
$Q(x) = 0$ iff $x = 0$, which implies $0 \notin G(x^*)$. Hence applying the upper semicontinuity of the set-valued
mapping $G$ one gets that there exists $a > 0$ and $\delta > 0$ such that
$$
  \inf_{y \in G(x)} |y| > a \quad \forall x \in \mathbb{R}^3 \colon |x - x^*| < \delta.
$$
Therefore for all sufficiently large $n$ one has $\inf\{|y| \colon y \in G(x^{(n)}) \} > a$, which contradicts the fact
that $|\nabla_u \omega(x^{(n)}, u)| \to 0$ as $n \to \infty$, since $\nabla_u \omega(x^{(n)}, u) \in G(x^{(n)})$ for 
all $n$.
\end{proof}

It remains to check that assumption 2 holds true, i.e. to verify that the closed-loop system 
(\ref{BI}), (\ref{ControlLaw}) has a solution for any initial data $x(0) \in \mathbb{R}^n \setminus C$ and 
$x(t) \notin C$ for all $t \in \mathbb{R}_+$. We prove a stronger assertion that, in particular, implies that for any 
$x(0) \notin C$ there exists a classical (i.e. continuously differentiable) solution of the closed-loop system.

\begin{proposition} \label{Prp_SolExist}
Let $\sigma(x(0)) \ne 0$. Then a solution $x(t)$ of the closed-loop system (\ref{BI}), (\ref{ControlLaw}) exists on
$\mathbb{R}_+$, $x(t) \notin C$ for all $t \ge 0$, and the control does not switch for any $t \in \mathbb{R}_+$
such that $x(t) \ne 0$. Thus, $x(t)$ is a classical solution of the system (\ref{BI}), (\ref{ControlLaw}) either on
$\mathbb{R}_+$ or on some finite time interval $[0, t_0)$, $t_0 > 0$. In the latter case, $x(t) \to 0$ as $t \to t_0$,
and $x(t)$ is an absolutely continuous solution of (\ref{BI}), (\ref{ControlLaw}) that is continuously differentiable on
$\mathbb{R}_+ \setminus \{ t_0 \}$.
\end{proposition}

\begin{proof}
Let $x_3(0) = 0$. Then the closed-loop system takes the form
$$
  \dot{x}_1 = - \gamma x_1 / \sigma(x), \quad
  \dot{x}_2 = - \gamma x_2 / \sigma(x), \quad
  \dot{x}_3 = 0.
$$
Hence, obviously, a solution of this system exists on $\mathbb{R}_+$, $x_3(t) = 0$ for any $t \in \mathbb{R}_+$,
and the control does not switch.

Let, now, $x_3(0) \ne 0$. Then the closed-loop system takes the form
\begin{gather}
  \dot{x}_1 = - \gamma |\nabla_u \omega(x, u)|^{-1} \frac{\partial \omega}{\partial u_1}(x, u), \quad 
  \dot{x}_2 = - \gamma |\nabla_u \omega(x, u)|^{-1} \frac{\partial \omega}{\partial u_2}(x, u), \label{CLSys1} \\ 
  \dot{x}_3 = - \gamma |\nabla_u \omega(x, u)|^{-1} \left( x_1 \frac{\partial \omega}{\partial u_2}(x, u) - x_2
  \frac{\partial \omega}{\partial u_1}(x, u)  \right), \label{CLSys2}
\end{gather}
where $\nabla_u \omega(x, u)$ has the form (\ref{PartDer1}), (\ref{PartDer2}). Clearly, a continuously differentiable
solution $x(t)$ of the system (\ref{CLSys1}), (\ref{CLSys2}) exists at least on some finite time interval. Denote by
$[0, t_0)$ the maximal interval of existence of this solution. Note that $x(t)$ is bounded on $[0, t_0)$ by virtue of
the facts that $Q$ is radially unbounded, and by the definition of the control law (\ref{GenContrL}) one has
$$
  \frac{d}{dt} Q(x(t), t) \le - \gamma |\nabla_u \omega(x, u)| < 0 \quad \forall t \in [0, t_0).
$$
Therefore, either $t_0 = +\infty$ and, thus, $x(t)$ is a continuously
differentiable solution of the closed-loop system that is defined and bounded on $\mathbb{R}_+$, and the control does
not switch, or at least one of the functions $x_3(t)$ and $\sigma(x(t))$ tends to zero as $t \to t_0$.

Let us show that $x_3(t) \to 0$ as $t \to t_0$ iff $\sigma(x(t)) \to 0$ as $t \to t_0$. In other words,
$t_0 < +\infty$ if and only if the trajectory $x(t)$ reaches the origin at $t = t_0$, which yields the required result.
Note also that if $t_0$ is finite, and $x(t) \to 0$ as $t \to t_0$, then the control switches to zero at time $t = t_0$
and $x(t) \equiv 0$ for any $t \ge t_0$. Thus, the solution $x(t)$ of the closed-loop system (\ref{BI}),
(\ref{ControlLaw}) is absolutely continuous on $\mathbb{R}_+$, and continuously differentiable on 
$\mathbb{R}_+ \setminus \{ t_0 \}$.

Let $x_3(0) > 0$, and suppose that $t_0 < +\infty$ and $x_3(t) \to 0$ as $t \to t_0$. The cases when $x_3(0) < 0$ or 
$\sigma(x(t)) \to 0$ as $t \to t_0$ can be considered in the same way. 

It is clear that $x_3(t) > 0$ for all $t \in [0, t_0)$. Choose $\varepsilon > 0$.  Then there exists $\delta > 0$ such
that $0 < x_3(t) < \varepsilon / 2$ for any $t \in [t_0 - \delta, t_0)$. Observe that for the closed-loop system one has
\begin{equation} \label{clloop_eq}
  \dot{x}_3 = - \frac{2 \gamma \sigma^2(x)}{|\nabla_u \omega(x, u)|} (2 x_3 - \sigma(x)), \quad
  \frac{d}{dt} \sigma(x) = - \frac{2 \gamma}{|\nabla_u \omega(x, u)|} (\sigma(x) - x_3).
\end{equation}
Therefore there exists $s \in [t_0 - \delta, t_0)$ such that $\sigma(s) < \varepsilon$, because otherwise from
(\ref{clloop_eq}) it follows that $\dot{x}_3(t) > 0$ on $[t_0 - \delta, t_0)$ or, equivalently, the function $x_3$ is
strictly increasing on $[t_0 - \delta, t_0)$, which contradicts the fact that $x_3(t) \to 0$ as $t \to t_0$.

Let us check that $\sigma(x(t)) < \varepsilon$ for any $t \in [s, t_0)$. Arguing by reductio ad absurdum, suppose that
there exists $\overline{t} \in (s, t_0)$ such that $\sigma(x(\overline{t})) \ge \varepsilon$.
Denote 
$$
  \tau = \inf \big\{ t \in (s, t_0) \colon \sigma(x(t)) = \varepsilon \big\}.
$$
Clearly, $\tau > s$, $\sigma(x(\tau)) = \varepsilon$ and for any $t \in [s, \tau)$ one has $\sigma(x(t)) < \varepsilon$.
Hence due to the continuity of $\sigma(x(t))$ there exists $\xi \in [s, \tau)$ such that 
$\varepsilon / 2 < \sigma(x(t)) < \varepsilon$ for all $t \in (\xi, \tau)$. Therefore with the use of (\ref{clloop_eq})
one gets that $\dot{\sigma}(x(t)) < 0$ on $(\xi, \tau)$, which implies that $\sigma(x(t))$ is strictly decreasing on
$(\xi, \tau)$ and, thus, $\sigma(x(\tau)) < \varepsilon$, which contradicts the definition of $\tau$. Hence
$\sigma(x(t)) < \varepsilon$ for any $t \in [s, t_0)$, which implies that $\sigma(x(t)) \to 0$ as $t \to t_0$, since
$\varepsilon > 0$ is arbitrary.
\end{proof}

Thus, the designed control law (\ref{ControlLaw}) stabilizes the Brockett integrator for any $\gamma > 0$ and any
initial point $x(0)$ that does not lie on the $x_3$-axis. Moreover, for any such initial point the control does not
switch and, thus, is continuous along the solutions of the closed-loop system. Therefore for any $x(0) \notin C$ there
exists a unique classical (i.e. continuously differentiable) solution of the closed-loop system. Finally, according to 
Remark~2 the maximum value of the control can be made arbitrarily small.

\noindent{\em Remark~5}.~Let us discuss the case when when the initial point lies on the $x_3$-axis. According to the
control algorithm(\ref{ControlLaw}) one chooses $u(x(0)) = \gamma v$ for an arbitrary $v \in \mathbb{R}^2$ such that
$|v| = 1$. The closed-loop system at this point takes the form
$$
  \dot{x}_1 = \gamma v_1 \quad
  \dot{x}_2 = \gamma v_2 \quad
  \dot{x}_3 = 0.
$$
Thus, according to the algorithm, the control input ``pushes the point off the $x_3$-axis'', and then the control
switches. However, it is not obvious whether an absolutely continuous solution of the closed-loop system (\ref{BI}),
(\ref{ControlLaw}) with the initial point $x(0)$ lying on the $x_3$-axis exists. If such a solution exists, then the
proposed control law (\ref{ControlLaw}) stabilizes the Brockett integrator for an arbitrary initial data.

Note that the set of limit points of the control law (\ref{ControlLaw}) as $\sigma(x) \to 0$ is the circle with radius
$\gamma$ centred at the origin, and $u(x) \to \gamma v$ as $\sigma(x) \to 0$, if the limit is taken along the ray 
$(\alpha v_1, \alpha v_2, x_3)$, $\alpha \ge 0$. Therefore, it is natural to expect that a solution of the closed-loop
system (\ref{BI}), (\ref{ControlLaw}) exists. However, the proof of the existence of a solution is outside the scope of
this article and is left for future research.

\subsection{Simulation}

Simulation of the closed-loop system with the following parameters was performed: $\gamma = 0.1$ and 
$x(0) = (0.2, 0.2, 0.2)$. Simulation results demonstrate convergence of the trajectory to the origin, 
see Fig.~\ref{Fig_BI} below.
\begin{figure}[h]
  \includegraphics[scale=0.65]{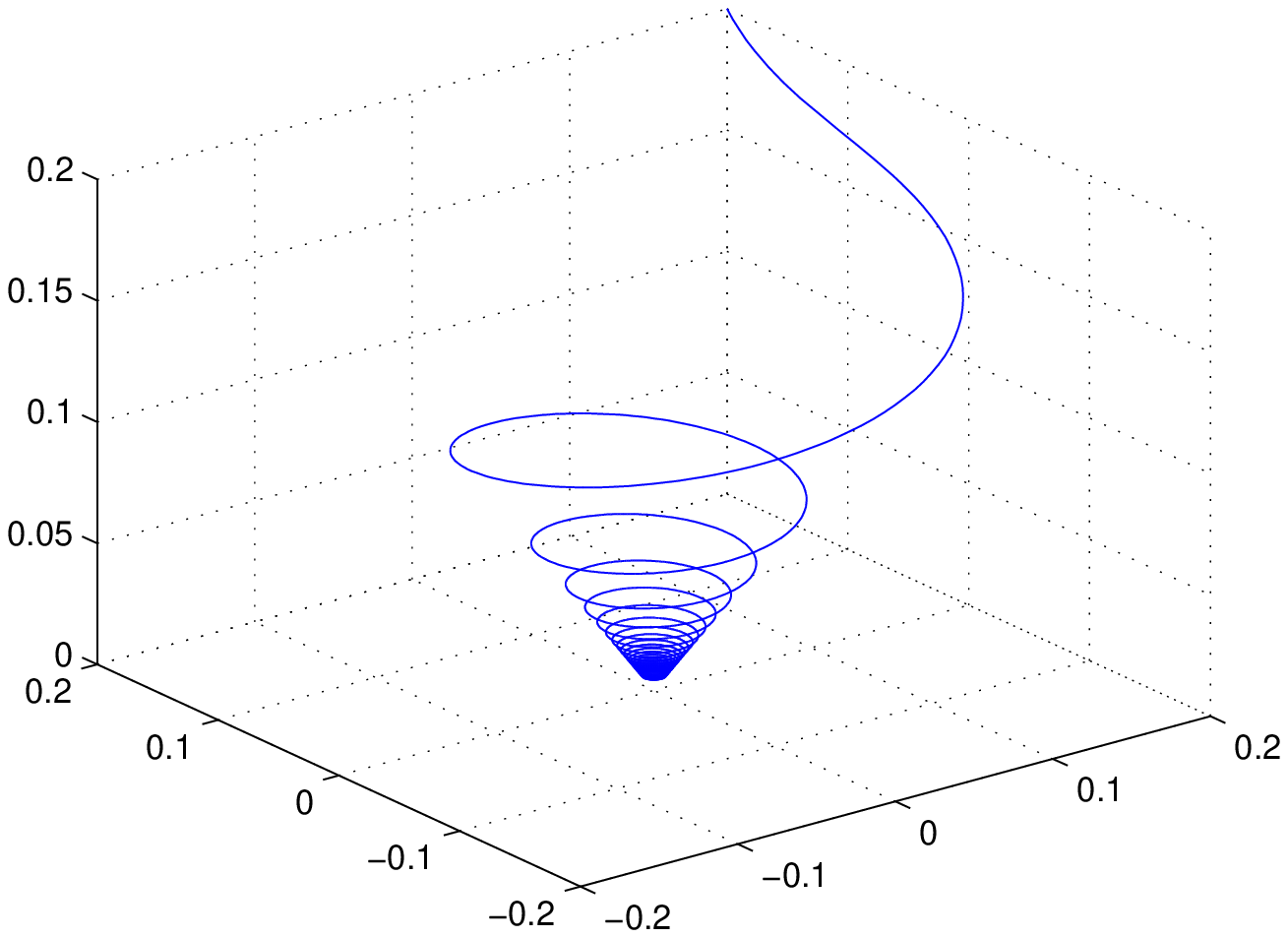}
  \includegraphics[scale=0.65]{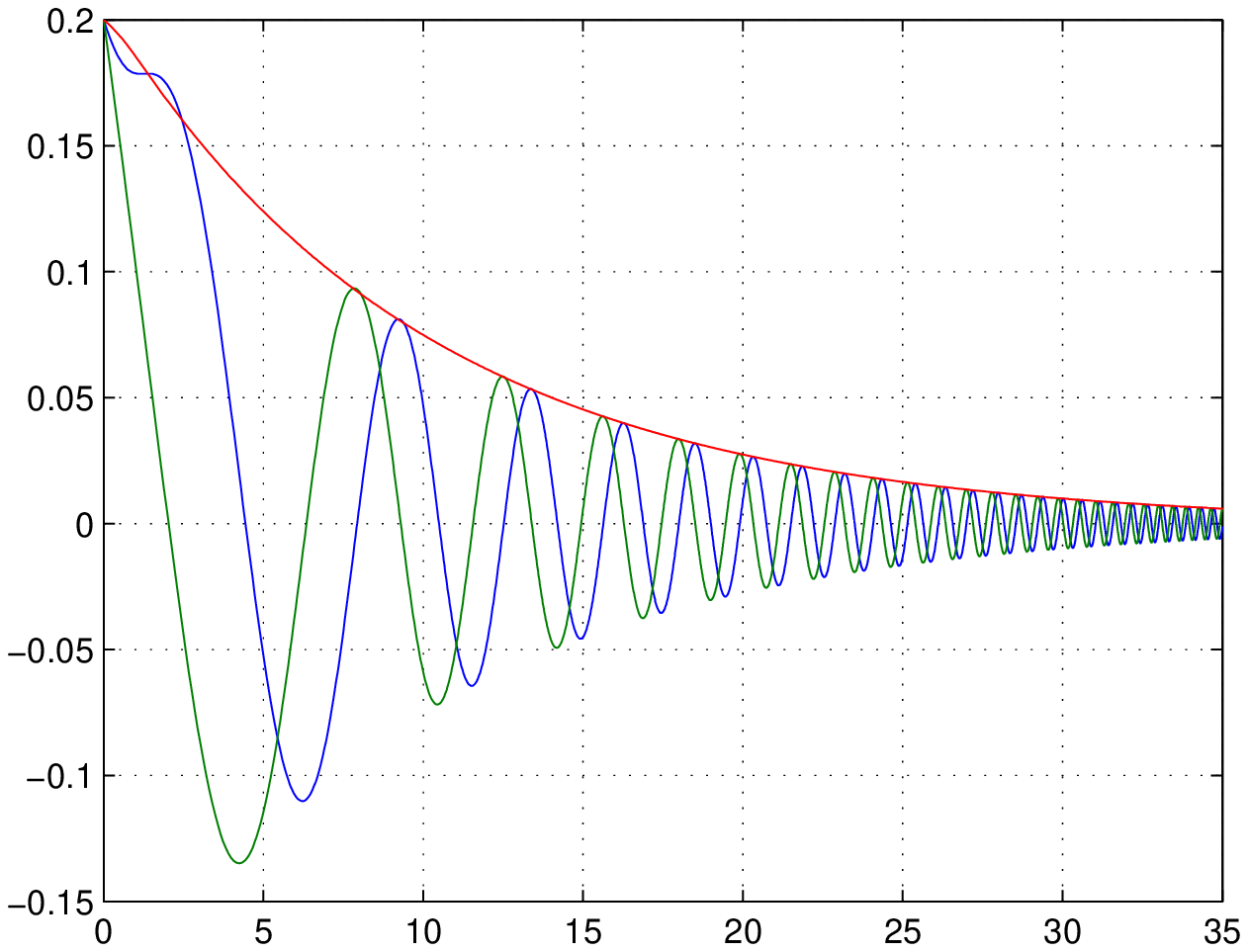}
  \caption{Simulation results of the closed-loop system}
  \label{Fig_BI}
\end{figure}

\section{Energy Control of a Vibrating String}
\label{Section_VibrString}

Theorem~\ref{Thm_LinCase} furnishes sufficient conditions for the convergence of the nonsmooth Speed-Gradient
algorithm~(\ref{PseudoSG}). However, in some important examples assumption 4 of this theorem is invalid. In this
section, we present such example, and demonstrate that even in this case one can prove that the control goal
$$
  \lim_{t \to \infty} Q(x(t), t) = 0
$$
is achieved with the use the Krasovskii-LaSalle invariance principle.

Consider an undamped vibrating string (see, e.g., \cite{Tufillaro,Bolwell}). The equation of motion of this string can
be written in the form
\begin{equation} \label{VibrString}
  \ddot{r} + \omega_0^2 (1 + K |r|^2) r = u,
\end{equation}
where $r = (x, y)$ is the displacement of an element of the string in the $xy$-plane that is perpendicular to
the string, $u = (u_1, u_2) \in \mathbb{R}^2$ represents the forcing term, $K > 0$ is a nonlinear coefficient that takes
into account the finite stretching of the string, and $\omega_0 = k \sqrt{T_0 / \mu}$ with $T_0$ being the average
tension of the string, $\mu$ its linear mass density and $k = 2 \pi / \lambda$, where $\lambda$ is the wavelength.

The system (\ref{VibrString}) can be written in the form
\begin{equation} \label{VibStr}
    \dot{q} = p, \quad
    \dot{p} = - \omega_0^2 \big( 1 + K (q_1^2 + q_2^2) \big) q + u,
\end{equation}
where $q = (q_1, q_2) = r$ and $p = (p_1, p_2) = \dot{r}$. The Hamiltonian for the system (\ref{VibStr}) has the form
$$
  H(q, p) = \frac{1}{2} (p_1^2 + p_2^2) + \frac{\omega_0^2}{2} (q_1^2 + q_2^2) + 
  \frac{\omega_0^2}{4} K (q_1^2 + q_2^2)^2.
$$
We pose the control problem as finding the control law $u = u(q, p)$, which ensures the objective
\begin{equation} \label{EnergyContGoal}
  H(q, p) \to H^* \quad \text{as} \quad t \to +\infty,
\end{equation}
where $H^* \ge 0$ is prespecified. Thus, the control objective is to reach a required energy level $H^*$. 

Introduce the following nonsmooth goal function
\begin{equation} \label{VibrStr_GoalFunc}
  Q(q, p) = |H(q, p) - H^*|.
\end{equation}
It is easy to see that the function $Q$ is locally Lipschitz continuous and Hadamard directionally differentiable.
For any $h \in \mathbb{R}^4$ its directional derivative has the form
\begin{multline*}
  Q'(q, p; h) = \sign(H(q, p) - H^*) \Big( \omega_0^2 \big( 1 + K (q_1^2 + q_2^2) \big) q_1 h_1 \\
  + \omega_0^2 \big( 1 + K (q_1^2 + q_2^2) \big) q_2 h_2 + p_1 h_3 + p_2 h_4 \Big)
\end{multline*}
in the case $H(q, p) \ne H^*$, and
$$
  Q'(q, p; h) = \Big| \omega_0^2 \big( 1 + K (q_1^2 + q_2^2) \big) q_1 h_1 + 
  \omega_0^2 \big( 1 + K (q_1^2 + q_2^2) \big) q_2 h_2 + p_1 h_3 + p_2 h_4 \Big|
$$
in the case $H(q, p) = H^*$. If $H(q, p) \ne H^*$, then the function $Q$ is differentiable, and we define
$$
  \omega(q, p, u) = Q'(q, p)^T F(q, p, u) = \sign(H(q, p) - H^*) \big( p_1 u_1 + p_2 u_2 \big),
$$
where $F(q, p, u)$ is the right-hand side of (\ref{VibStr}). In the case $H(q, p) = H^*$, there is no linear in $u$
function $\omega(q, p, u)$ such that
$$
  Q'(q, p; F(q, p, u)) \le \omega(q, p, u) \quad \forall q, p, u \in \mathbb{R}^2.
$$
However, taking into account Remark~3 we define $\omega(q, p, u) \equiv 0$ for any $q$ and $p$ such that 
$H(q, p) = H^*$. Then according to the nonsmooth Speed-Gradient algorithm we define the control law as follows
$$
  u = - \gamma \nabla_u \omega(q, p, u).
$$
Thus, the control law has the form
\begin{equation} \label{EnergyControlLaw}
  u(q, p) = - \gamma \sign(H(q, p) - H^*) p,
\end{equation}
where $\sign(0) = 0$.

It is easy to see that assumption~4 of Theorem~\ref{Thm_LinCase} is not satisfied in the example under consideration.
Therefore, one cannot apply Theorem~\ref{Thm_LinCase} in order to prove the achievement of the control goal
(\ref{EnergyContGoal}). However, the following result holds true.

\begin{proposition}
For any $\gamma > 0$ and $q(0), p(0) \in \mathbb{R}^2$ such that $|q(0)| + |p(0)| \ne 0$ and $H(q(0), p(0)) \ne H^*$
there exist an absolutely continuous solution $(q(t), p(t))$ of the system (\ref{VibStr}), (\ref{EnergyControlLaw}) that
is defined and bounded on $\mathbb{R}_+$, and the control goal (\ref{EnergyContGoal}) is achieved. Moreover, if 
$H^* \ne 0$, then there exists $T > 0$ such that
\begin{equation} \label{ControlGoal_VibStr}
  H(q(t), p(t)) \to H^* \quad \text{as} \quad t \to T,
\end{equation}
i.e. the control goal is achieved in finite time.
\end{proposition}

\begin{proof}
Note that a solution $(q(t), p(t))$ of the system (\ref{VibStr}), (\ref{EnergyControlLaw}) exists at least on some
finite time interval. Furthermore, from the equality
$$
  \frac{d}{dt} H(q(t), p(t)) = - \gamma \sign(H(q(t), p(t)) - H^*) |p|^2
$$
it follows that $H(q(t), p(t)) \le \max\{ H(q(0), p(0)), H^* \}$, which implies that $(q(t), p(t))$ is bounded. Hence
either $H(q(t), p(t)) \ne H^*$ for all $t$, and $(q(t), p(t))$ is a continuously differentiable solution of the system
(\ref{VibStr}), (\ref{EnergyControlLaw}) that is defined and bounded on $\mathbb{R}_+$ or there exists some $T \ge 0$
such that $H(q(t), p(t)) \to H^*$ as $t \to T$. In the latter case, the control switches to zero at time $t = T$, and 
$(q(t), p(t))$ coincides with a solution of the system (\ref{VibStr}) with $u = 0$, since the total energy $H(q, p)$ is
conserved along solutions of the unforced system (\ref{VibStr}). Thus, $(q(t), p(t))$ is an absolutely continuous
solution $(q(t), p(t))$ of the system (\ref{VibStr}), (\ref{EnergyControlLaw}) that is defined and bounded on
$\mathbb{R}_+$. It remains to show that the control goal (\ref{EnergyContGoal}) is achieved, and in the case $H^* \ne 0$
there exists $T > 0$ such that $H(q(t), p(t)) \to H^*$ as $t \to T$.

Suppose, at first, that $H(q(0), p(0)) > H^*$. Then the closed-loop system takes the form
\begin{equation} \label{VibStr_FirstCase}
  \dot{q} = p, \quad
  \dot{p} = - \omega_0^2 \big( 1 + K (q_1^2 + q_2^2) \big) q - \gamma p.
\end{equation}
Clearly, there exists a unique solution $(q_0(t), p_0(t))$ of the system (\ref{VibStr_FirstCase}) satisfying 
$q_0(0) = q(0)$ and $p_0(0) = p(0)$ that is defined and bounded on $\mathbb{R}_+$ by virtue of the fact that
$$
  \frac{d}{dt} H(q_0(t), p_0(t)) = - \gamma |p_0(t)|^2 \le 0 \quad \forall t \in \mathbb{R}_+.
$$
Applying the Krasovskii-LaSalle invariance principle to the system (\ref{VibStr_FirstCase}) with $H(q, p)$ as a
Lyapunov function one obtains that $(q_0(t), p_0(t)) \to (0, 0)$ as $t \to \infty$. Consequently, 
$H(q_0(t), p_0(t)) \to 0$ as $t \to \infty$. Therefore $H(q(t), p(t)) \to 0$ as $t \to \infty$ in the case $H^* = 0$,
and $H(q(t), p(t)) \to H^*$ as $t \to T$ for some $T > 0$ in the case $H^* > 0$, where $(q(t), p(t))$ is a solution of
the closed-loop system (\ref{VibStr}), (\ref{EnergyControlLaw}).

Suppose, now, that $H(q(0), p(0)) < H^*$. Then the closed-loop system takes the form
\begin{equation} \label{VibStr_SecondCase}
  \dot{q} = p, \quad
  \dot{p} = - \omega_0^2 \big( 1 + K (q_1^2 + q_2^2) \big) q + \gamma p.
\end{equation}
Denote by $(q_0(t), p_0(t))$ a unique solution of the system (\ref{VibStr_SecondCase}) satisfying $q_0(0) = q(0)$ and
$p_0(0) = p(0)$ that exists at least on some finite time interval. Also, denote by $[0, t_0)$ the maximal interval of
existence of this solution. Observe that
\begin{equation} \label{LowerEstimate_VibrStr}
  \frac{d}{dt} H(q_0(t), p_0(t)) = \gamma |p_0(t)|^2 \ge 0 \quad \forall t \in [0, t_0).
\end{equation}
Hence the function $H(q_0(\cdot), p_0(\cdot))$ is nondecreasing.

Note that if $t_0 < + \infty$, then $H(q_0(t), p_0(t)) \to \infty$ as $t \to t_0$. Therefore, either there exists
$T \in (0, t_0)$ such that $H(q_0(t), p_0(t)) < H^*$ for all $t \in [0, T)$ and $H(q_0(T), p_0(T)) = H^*$ or 
$$
  \lim_{t \to t_0} H(q_0(t), p_0(t)) < H^*.
$$
In the former case one has $H(q(t), p(t)) \to H^*$ as $t \to T$, where $(q(t), p(t))$ is a solution of (\ref{VibStr}),
(\ref{EnergyControlLaw}), while in the latter case one has $t_0 = + \infty$ and $H(q_0(t), p_0(t)) < H^*$ for all 
$t \ge 0$. Let us show that the latter case is impossible.

Indeed, denote 
$$
  D = \{ (q_0(t), p_0(t)) \colon t \ge 0 \}.
$$
Clearly, $D$ is a bounded invariant set of the system (\ref{VibStr_SecondCase}). Observe that the function 
$V(q, p) = H^* - H(q, p)$ is nonnegative and continuous on the set $D$, and its derivative along solutions of the system
(\ref{VibStr_SecondCase}) has the form
$$
  \frac{d}{dt} V(q, p) = - \gamma |p|^2 \le 0.
$$
Hence $V$ is a Lyapunov function of the system
(\ref{VibStr_SecondCase}) on the invariant set $D$. Therefore, by the Krasovskii-LaSalle invariance principle
(see~\cite{LaSalle}, Theorem~6.4) any solution of (\ref{VibStr_SecondCase}) starting in $D$ converges to the largest
invariant set of (\ref{VibStr_SecondCase}) in the set
$$
  E = \Big\{ (q, p) \in \cl D \colon \frac{d}{dt} V(q, p) = 0 \Big\},
$$
where $\cl D$ is the closure of the set $D$. Note that $d V(q, p) / dt = 0$ iff $p = 0$, and the only invariant
set of the system (\ref{VibStr_SecondCase}) contained in the set $\{ (q, p) \in \mathbb{R}^4 \colon p = 0 \}$ is the
equilibrium point $(0, 0)$. Therefore any solution of (\ref{VibStr_SecondCase}) starting in $D$ must converge to the
origin. In particular, one has $(q_0(t), p_0(t)) \to (0, 0)$ as $t \to \infty$. However, from the estimate
(\ref{LowerEstimate_VibrStr}) and the fact that $|q(0)| + |p(0)| \ne 0$ it follows that
$$
  H(q_0(t), p_0(t)) \ge H(q(0), p(0)) > 0,
$$
which contradicts the fact that $(q_0(t), p_0(t)) \to (0, 0)$ as $t \to \infty$. Thus, the case when 
$H(q_0(t), p_0(t)) < H^*$ for all $t \in \mathbb{R}_+$ is impossible, which completes the proof.
\end{proof}

\noindent{\em Remark~6}.~Note that one can use the smooth goal function
$$
  Q(q, p) = \frac{1}{2} \big( H(q, p) - H^* \big)^2
$$
instead of the nonsmooth goal function~(\ref{VibrStr_GoalFunc}), and apply the standard Speed-Gradient algorithm in
order to design a control law for the problem under consideration. However, there are no results on the finite-time
convergence of the smooth Speed-Gradient algorithm. In contrast, the use of the nonsmooth Speed-Gradient algorithm
allows one to guarantee finite-time convergence to any nonzero energy level.

\section{Conclusion}

It is well known that relaxation of smoothness condition may greatly improve performance of the control systems (recall
variable structure systems). However, many specific problems are still to be examined. The contribution of this paper
is twofold. On the one hand, we propose a nonsmooth extension of the Speed-Gradient algorithms that, in turn, extend
the classical $L_gV$ control. On the other hand, we present yet another almost global stabilizer for Brockett
integrator. Its additional features are possibility of stabilization with arbitrarily small control level and
continuity of the control along trajectories of the closed-loop system.

An avenue for further research is testing nonsmooth Speed-Gradient algorithms for various nonlinear control problems.

\bibliographystyle{abbrv}  
\bibliography{BI_bibl}

\begin{thebibliography}{10}

\bibitem{Agrachev}
A.~A. Agrachev and Y.~L. Sachkov.
\newblock {\em Control Theory from the Geometric Viewpoint}.
\newblock Springer, Berlin, 2004.

\bibitem{Bonnard}
A.~Anzaldo-{M}enese, F.~Monroy-{P}\'erez, B.~Bonnard, and J.-P. Gauthier,
  editors.
\newblock {\em Contemporary Trends in Nonlinear Geometric Control Theory and
  Its Applications}, Singapore, 2002. World Scientific.

\bibitem{Astolfi}
A.~Astolfi.
\newblock Discontinuous control of nonholonomic systems.
\newblock {\em Systems and Control Letters}, 27:37--45, 1996.

\bibitem{Astolfi2}
A.~Astolfi.
\newblock Discontinuous control of the brockett integrator.
\newblock In {\em Proceedings of the 36th IEEE Conference on Decision and
  Control, San Diego, California, USA}, pages 4334--4339, 1997.

\bibitem{Bloch}
A.~Bloch and S.~Drakunov.
\newblock Stabilization and tracking in the nonholonomic integrator via sliding
  modes.
\newblock {\em Systems and Control Letters}, 29:91--99, 1996.

\bibitem{BlochDrakunovKinyon}
A.~M. Bloch, S.~V. Drakunov, and M.~K. Kinyon.
\newblock Stabilization of nonholonomic systems using isospectral flows.
\newblock {\em SIAM J. Control Optim.}, 38:855--874, 2000.

\bibitem{Bolwell}
J.~E. Bolwell.
\newblock The flexible string's neglected term.
\newblock {\em J. Sound Vib.}, 206:618--623, 1997.

\bibitem{Brockett}
R.~W. Brockett.
\newblock Asymptotic stability and feedback stabilization.
\newblock In R.~W. Brockett, R.~S. Millman, and H.~J. Sussmann, editors, {\em
  Differential Geometric Control Theory}, pages 181--191. Birkhauser, Boston,
  1983.

\bibitem{Brockett_GeomCont}
R.~W. Brockett.
\newblock The early days of geometric nonlinear control.
\newblock {\em Automatica}, 50:2203--2224, 2014.

\bibitem{Bullo}
F.~Bullo and A.~D. Lewis.
\newblock {\em Geometric Control of Mechanical Systems}.
\newblock Springer, New York, 2005.

\bibitem{Clarke}
F.~Clarke.
\newblock Discontinuous feedback and nonlinear systems.
\newblock In {\em Proceedings of IFAC Conf. Nonlinear Control (NOLCOS),
  Bologna}, pages 1--29, 2010.

\bibitem{Demyanov}
V.~F. Demyanov and A.~M. Rubinov.
\newblock {\em Constructive Nonsmooth Analysis}.
\newblock Peter Lang, Frankfurt am Main, 1995.

\bibitem{DF}
M.~V. Dolgopolik and A.~L. Fradkov.
\newblock Nonsmooth speed-gradient algorithms.
\newblock In {\em Proceedings of 15th European Control Coference, Linz}, 2015.
\newblock Submitted for publication.

\bibitem{F79}
A.~L. Fradkov.
\newblock Speed-gradient scheme and its application in adaptive control
  problems.
\newblock {\em Automation and Remote Control}, 40:1333--1342, 1979.

\bibitem{F86}
A.~L. Fradkov.
\newblock Integro-differentiating algorithms of speed gradient.
\newblock {\em Doklady Akademii Nauk SSSR}, 286:832--835, 1986.

\bibitem{FGHP95}
A.~L. Fradkov, P.~Y. Guzenko, D.~Hill, and A.~Y. Pogromsky.
\newblock Speed gradient control and passivity of nonlinear oscillators.
\newblock In {\em IFAC Symposium on Nonlinear Control Systems, NOLCOS'95, Tahoe
  City, USA}, pages 655--659, 1995.

\bibitem{FMN99}
A.~L. Fradkov, I.~V. Miroshnik, and V.~O. Nikiforov.
\newblock {\em Nonlinear and Adaptive Control of Complex Systems}.
\newblock Kluwer Academic Publishers, Dordrecht, 1999.

\bibitem{FP96}
A.~L. Fradkov and A.~Y. Pogromsky.
\newblock Speed-gradient control of chaotic continuous-time systems.
\newblock {\em IEEE Trans.Circuits and Systems, part I}, 43:907--913, 1996.

\bibitem{FS92}
A.~L. Fradkov and A.~A. Stotsky.
\newblock Speed-gradient adaptive control algorithms for mechanical systems.
\newblock {\em Int. J. Adaptive Control and Signal Processing}, 6:211--220,
  1992.

\bibitem{Morse}
J.~P. Hespanha and A.~S. Morse.
\newblock Stabilization of nonholonomic integrators via logic-based switching.
\newblock {\em Automatica}, 35:385--393, 1999.

\bibitem{JQ78}
V.~Jurdjevic and J.~P. Quinn.
\newblock Controllability and stability.
\newblock {\em J. Differential Equations}, 28:381--389, 1978.

\bibitem{Khennouf}
H.~Khennouf and C.~{Canudas~de~Wit}.
\newblock On the construction of stabilizing discontinuous controllers for
  nonholonomic systems.
\newblock In {\em Symposium on Nonlinear Control System Design, Lake Tahoe,
  CA}, pages 747--752, 1995.

\bibitem{Kovalev}
A.~M. Kovalev and V.~N. Nespirnyy.
\newblock Impulsive discontinuous stabilization of the brockett integrator.
\newblock {\em Journal of Computer and Systems Sciences International},
  44:671--681, 2005.

\bibitem{LaSalle}
J.~P. LaSalle.
\newblock {\em The Stability of Dynamical Systems}.
\newblock SIAM, Philadelphia, Pa., 2002.

\bibitem{Prieur}
C.~Prieur and E.~Tr\'elat.
\newblock Robust optimal stabilization of the brockett integrator via a hybrid
  feedback.
\newblock {\em Math. Control Signals Systems}, 17:201--216, 2005.

\bibitem{Prieur2}
C.~Prieur and E.~Tr\'elat.
\newblock Quasi-optimal robust stabilization of control systems.
\newblock {\em SIAM J. Control Optim.}, 45:1875--1897, 2006.

\bibitem{Ryan}
E.~P. Ryan.
\newblock On brockett's condition for smooth stabilizability and its necessity
  in a context of nonsmooth feedback.
\newblock {\em SIAM J. Control Optim.}, 32:1597--1604, 1994.

\bibitem{Selivanov12}
A.~A. Selivanov, J.~Lehnert, T.~Dahms, P.~Hovel, A.~L. Fradkov, and E.~Schoell.
\newblock Adaptive synchronization in delay-coupled networks of stuart-landau
  oscillators.
\newblock {\em Phys.Rev. E}, 85:016201, 2012.

\bibitem{Stefani}
G.~Stefani, U.~Boscain, J.-P. Gauthier, A.~Sarychev, and M.~Sigalotti, editors.
\newblock {\em Geometric Control Theory and Sub-Riemannian Geometry}, Cham,
  2014. Springer International Publ.

\bibitem{Tufillaro}
N.~B. Tufillaro.
\newblock Nonlinear and chaotic string vibrations.
\newblock {\em Am. J. Phys.}, 57:408--414, 1989.

\bibitem{Vdovin}
S.~A. Vdovin, A.~M. Taras'yev, and V.~N. Ushakov.
\newblock Construction of the attainability set of a brockett integrator.
\newblock {\em Journal of Applied Mathematics and Mechanics}, 68:631--646,
  2004.

\end{thebibliography}

\end{document}